\documentclass[12pt,reqno]{amsart}

\usepackage{mathdots}

\usepackage{color}

\usepackage{pdfsync}

\usepackage{enumitem}

\usepackage{wasysym}

\usepackage{amssymb}

\usepackage{mathrsfs}

\DeclareMathAlphabet{\mathpzc}{OT1}{pzc}{m}{it}

\usepackage[all]{xy}

\usepackage{tikz}
\usetikzlibrary{arrows,decorations.pathmorphing,decorations.pathreplacing,positioning,shapes.geometric,shapes.misc,decorations.markings,decorations.fractals,calc,patterns}

\usepackage{float}

\usepackage[bottom]{footmisc}

\entrymodifiers={+!!<0pt,\fontdimen22\textfont2>}

\def\cC{\mathscr{C}}
\def\cD{\mathscr{D}}

\def\cM{\mathscr{M}}

\def\cT{\mathscr{T}}
\def\cU{\mathscr{U}}

\def\cX{\mathscr{X}}

\def\BZ{\mathbb{Z}}

\def\add{\operatorname{add}}

\def\adots{\mathinner{\mkern1mu\raise1.0pt\vbox{\kern7.0pt\hbox{.}}\mkern2mu\raise4.0pt\hbox{.}\mkern2mu\raise7.0pt\hbox{.}\mkern1mu}}

\def\Coker{\operatorname{Coker}}

\def\dddots{\mathinner{\mkern1mu\raise10.0pt\vbox{\kern7.0pt\hbox{.}}\mkern2mu\raise5.3pt\hbox{.}\mkern2mu\raise1.0pt\hbox{.}\mkern1mu}}
\def\dddotssmall{\mathinner{\mkern1mu\raise7.0pt\vbox{\kern7.0pt\hbox{.}}\mkern-1mu\raise4pt\hbox{.}\mkern-1mu\raise1.0pt\hbox{.}\mkern1mu}}

\def\Db{\cD^{\operatorname{b}}}

\def\dual{\operatorname{D}}
\def\End{\operatorname{End}}

\def\Ext{\operatorname{Ext}}

\def\Fac{\operatorname{Fac}}

\def\gldim{\operatorname{gldim}}

\def\H{\operatorname{H}}

\def\Hom{\operatorname{Hom}}
\def\id{\operatorname{id}}

\def\Ker{\operatorname{Ker}}

\def\mod{\mathsf{mod}}

\def\opp{\operatorname{op}}

\def\SL2{\operatorname{SL}_2}

\def\Tr{\operatorname{Tr}}
\def\Uexact{\cU\mbox{\rm -exact}}

\def\Xexact{\cX\mbox{\rm -exact}}

\numberwithin{equation}{section}

\newtheorem{Lemma}{Lemma}[section]
\newtheorem{Theorem}[Lemma]{Theorem}
\newtheorem{Proposition}[Lemma]{Proposition}

\theoremstyle{definition}
\newtheorem{Definition}[Lemma]{Definition}
\newtheorem{Setup}[Lemma]{Setup}

\newtheorem{Remark}[Lemma]{Remark}

\newtheorem{Example}[Lemma]{Example}

\newtheorem*{bfhpg*}{}

\begin{document}

\setlength{\parindent}{0pt}
\setlength{\parskip}{7pt}

\title[Higher torsion]{Torsion classes and $\rm
t$-structures in higher homological algebra}

\author{Peter J\o rgensen}
\address{School of Mathematics and Statistics,
Newcastle University, Newcastle upon Tyne NE1 7RU, United Kingdom}
\email{peter.jorgensen@ncl.ac.uk}
\urladdr{http://www.staff.ncl.ac.uk/peter.jorgensen}


\keywords{Aisle, cluster tilting module, cluster tilting subcategory,
  intermediate t-structure, $n$-abelian category, $n$-angulated
  category, $n$-exact sequence, Wakamatsu's Lemma, weakly
  $n$-representation finite algebra}

\subjclass[2010]{16G10, 16S90, 18E10, 18E30, 18E40}


\begin{abstract} 

Higher homological algebra was introduced by Iyama.  It is also known
as $n$-homological algebra where $n \geqslant 2$ is a fixed integer,
and it deals with $n$-cluster tilting subcategories of abelian
categories.

\medskip
\noindent
All short exact sequences in such a subcategory are split, but it has
nice exact sequences with $n+2$ objects.  This was recently formalised
by Jasso in the theory of $n$-abelian categories.  There is also a
derived version of $n$-homological algebra, formalised by Geiss,
Keller, and Oppermann in the theory of $( n+2 )$-angulated categories
(the reason for the shift from $n$ to $n+2$ is that angulated
categories have {\em tri}an\-gu\-la\-ted categories as the ``base
case''). 

\medskip
\noindent
We introduce torsion classes and t-structures into the theory of
$n$-abelian and $( n+2 )$-angulated categories, and prove several
results to motivate the definitions.  Most of the results concern the
$n$-abelian and $( n+2 )$-angulated categories $\cM( \Lambda )$ and
$\cC( \Lambda )$ associated to an $n$-representation finite algebra
$\Lambda$, as defined by Iyama and Oppermann.  We characterise torsion
classes in these categories in terms of closure under higher
extensions, and give a bijection between torsion classes in $\cM(
\Lambda )$ and intermediate t-structures in $\cC( \Lambda )$ which is
a category one can reasonably view as the $n$-derived category of
$\cM( \Lambda )$.  We hint at the link to $n$-homological tilting
theory.

\end{abstract}

\maketitle

\setcounter{section}{-1}
\section{Introduction}
\label{sec:introduction}

Higher homological algebra was introduced and developed by Iyama in
\cite{I2}, \cite{I1}, \cite{I3}.  It is also known as $n$-homological
algebra where $n \geqslant 2$ is an integer.

This paper introduces torsion classes and t-structures into
$n$-homological algebra.  The t-structures will occur in the form of
so-called aisles.  We will prove several results to motivate the
definitions.  The results are mainly foundational, only hinting at the
link to $n$-homological tilting theory which was initiated in
\cite{IO} and \cite{OT}.  However, as this theory is developed
further, we expect that, like classic tilting, it will turn out to be
closely linked to torsion classes.

$n$-homological algebra concerns so-called $n$-cluster tilting
subcategories of abelian categories, see \cite[def.\ 1.1]{I1} or
Definition \ref{def:cluster_tilting}.  All short exact sequences in
such a subcategory are split, but it has well-behaved exact sequences
with $n+2$ objects, also known as $n$-exact sequences.  There is also
a derived version of the theory focusing on $n$-cluster tilting
subcategories of triangulated categories as introduced in \cite[sec.\
5.1]{KR}.  There are rich examples of $n$-cluster tilting
sub\-ca\-te\-go\-ri\-es, many with interesting links to combinatorics
and geometry, see \cite{HI1}, \cite{HI2}, \cite{IO}, \cite{OT}.  The
abelian version of $n$-homological algebra was recently formalised to
the theory of $n$-abelian categories by Jasso in \cite[def.\ 3.1]{J},
see Definition \ref{def:n-abelian}.  In such a category, each morphism
$x \rightarrow y$ has not a kernel, but an $n$-kernel, that is, there
is a sequence $m_n \rightarrow \cdots \rightarrow m_1
\rightarrow x \rightarrow y$ with certain properties.  Similarly, $x
\rightarrow y$ has an $n$-cokernel and a list of axioms are satisfied.
The derived version of $n$-homological algebra was formalised to the
theory of $( n+2 )$-angulated categories by Geiss, Keller, and
Oppermann in \cite[def.\ 2.1]{GKO}, see Definition
\ref{def:n-angulated}.

Let $n \geqslant 1$ be a fixed integer for the rest of the paper (the
case $n = 1$ is included since $1$-homological algebra makes sense and
reduces to classic homological algebra).  The following is our main
definition.  The $n$-suspension functor of an $( n+2 )$-angulated
category is denoted by $\Sigma^n$ but is not required actually to be
an $n^{\rm th}$ power.

\begin{Definition}
[Torsion classes]
\label{def:torsion}
Let $\cM$ be an $n$-abelian category.  A full subcategory $\cU
\subseteq \cM$ is called a {\em torsion class} if, for each $m \in
\cM$, there is an $n$-exact sequence
\[
  0 \rightarrow u \stackrel{ \theta }{ \rightarrow } m
    \rightarrow v^1 \rightarrow \cdots \rightarrow v^n \rightarrow 0
\]
where $u \in \cU$ and $v^1 \rightarrow \cdots \rightarrow v^n$ is in
the class $\Uexact$ defined in Equation \eqref{equ:exact} below.

Let $\cC$ be an $( n+2 )$-angulated category.  A full subcategory $\cX
\subseteq \cC$ is called a {\em torsion class} if it is
closed under sums and summands and, for each $c \in \cC$, there is a
$( n+2 )$-angle
\[
  x \stackrel{ \xi }{ \rightarrow } c
    \rightarrow y^1 \rightarrow \cdots \rightarrow y^n
    \rightarrow \Sigma^n x
\]
where $x \in \cX$ and $y^1 \rightarrow \cdots \rightarrow y^n$ is in
the class $\Xexact$. 
\hfill $\Box$
\end{Definition}

Here $\Uexact$ is defined by
\begin{equation}
\label{equ:exact}
  \Uexact
  = \Biggl\{
      \begin{array}{ll}
        v^1 \rightarrow \cdots \rightarrow v^n \mbox{ is } \\
        \mbox{a complex in } \cM
      \end{array}
    \Bigg|
      \begin{array}{ll}
        0
        \rightarrow \cM( u,v^1 )
        \rightarrow \cdots
        \rightarrow \cM( u,v^n )
        \rightarrow 0 \\
        \mbox{is exact for each } u \in \cU
      \end{array}
    \Biggr\}
\end{equation}
and $\Xexact$ is defined similarly.  Let us make a few observations:
Definition \ref{def:torsion} implies that $\cU$ is closed under sums
and summands, see Lemma \ref{lem:precovers}(iii).  While $\Uexact$ and
$\Xexact$ depend on $n$, we omit $n$ from the notation since it is
fixed throughout the paper.  If $n = 1$ then $\Uexact = \cU^{ \perp }$
so the definition specialises to the usual definitions of torsion
classes in abelian and triangulated categories, see \cite[def.\ 1.1
and prop.\ 1.5]{ASS} and \cite[def.\ 2.2]{IY}.  We will prove a number
of results to motivate Definition \ref{def:torsion}.  The easiest to
state is the following immediate consequence of Lemma
\ref{lem:Wakamatsu} below.

\begin{Theorem}
\label{thm:n-angulated_torsion2}
Let $\cC$ be an $( n+2 )$-angulated category, $\cX \subseteq \cC$ a
full subcategory closed under sums and summands which is covering and
left closed under $n$-extensions.  Then $\cX$ is a torsion class.
\end{Theorem}

For $\cX$ to be {\em left closed under $n$-extensions} means that, for
each morphism $x'' \stackrel{ \delta }{ \rightarrow } \Sigma^n x'$
with $x', x'' \in \cX$, there is an $( n+2 )$-angle $x' \rightarrow
c^1 \rightarrow c^2 \rightarrow \cdots \rightarrow c^n \rightarrow x''
\stackrel{ \delta }{ \rightarrow } \Sigma^n x'$ with $c^1 \in \cX$.
Note that this is strictly weaker than the condition that there is an
$( n+2 )$-angle with each $c^i$ in $\cX$; see Example
\ref{exa:extension_closed}.  The theorem is an $n$-homological version
of one implication in Iyama and Yoshino's characterisation of torsion
classes in triangulated categories, see \cite[prop.\ 2.3(1)]{IY}.  In
a more special setting, we are able to give a necessary and sufficient
condition for $\cX$ to be a torsion class.  Let $\Lambda$ be a weakly
$n$-representation finite algebra with $\gldim( \Lambda ) \leqslant n$
and let $\cC( \Lambda )$ be the associated $(n+2)$-angulated category,
see \cite[def.\ 2.2]{IO}, \cite[thms.\ 1.6 and 1.21]{I1}, \cite[thm.\
1]{GKO} or Setup \ref{set:blanket}.

\begin{Theorem}
[=Theorem \ref{thm:n-angulated_torsion}]
\label{thm:A}
Let $\cX \subseteq \cC( \Lambda )$ be a full subcategory closed under
sums and summands.  Then
$\cX$ is a torsion class in $\cC( \Lambda )$
$\Leftrightarrow$
$\cX$ is left closed under $n$-extensions.
\end{Theorem}

The reason we can dispense with the explicit condition that $\cX$ is
covering is that this is now automatic, see Lemma
\ref{lem:functorially_finite}(ii).  We are able to get implications in
both directions because we here have access to ``minimal''
$(n+2)$-angles, see the proof of Theorem
\ref{thm:n-angulated_torsion}.

The algebra $\Lambda$ also has an associated $n$-abelian category
$\cM( \Lambda )$, see \cite[def.\ 2.2]{IO}, \cite[thm.\ 1.6]{I1},
\cite[thm.\ 3.16]{J} or Setup \ref{set:blanket}, and we give the
following characterisation of torsion classes in $\cM( \Lambda )$.

\begin{Theorem}
[=Theorem \ref{thm:n-abelian_torsion}]
\label{thm:B}
Let $\cU \subseteq \cM( \Lambda )$ be a full subcategory closed under
sums and summands.  Then
$\cU$ is a torsion class in $\cM( \Lambda )$
$\Leftrightarrow$
$\cU$ has Property $(E)$.
\end{Theorem}

Property $(E)$ is stated in Definition \ref{def:E} and we do not
reproduce it here, but point out that it is an $n$-homological version
of the classic property of being closed under quotients and
extensions, to which it specialises when $n = 1$, see Proposition
\ref{pro:1}.  The following provides a simple example.

\begin{Theorem}
[=Theorem \ref{thm:Fac}]
\label{thm:06}
Let $t$ be an $n$-APR tilting module for $\Lambda$, as introduced in
\cite[def.\ 3.1 and obs.\ 4.1(1)]{IO}.  Then $\cM( \Lambda ) \cap
\Fac(t)$ is a splitting torsion class in $\cM( \Lambda )$.
\end{Theorem}

{\em Splitting} means that the morphism $\theta$ in Definition
\ref{def:torsion} can be taken to be a split monomorphism, and
$\Fac(t)$ is the class of quotients in $\mod( \Lambda )$ of modules of
the form $t^{ \oplus d }$.  For an example of a torsion class in an $(
n+2 )$-angulated category $\cC$, take $\cX = \add( T )$ where $T$ is a
cluster tilting object in $\cC$ in the sense of \cite[def.\ 5.3]{OT},
and $\add$ means taking summands of sums.

In classic homological algebra, there is a well-known bijection
between torsion classes in an abelian category and so-called
intermediate t-structures in the derived category.  This is due to
\cite[thm.\ 3.1]{BR} and \cite[prop.\ 2.1]{HRS}, see \cite[prop.\
2.1]{W} for a clean statement.  Recall that a classic t-structure
consists of two classes of objects, the aisle and the coaisle, and
that the aisle is a torsion class, see \cite[sec.\ 1]{KV}.  We will
consider aisles in $( n+2 )$-angulated categories, and will use the
previous results to show the following.  Note that $\cC( \Lambda )$
can reasonably be thought of as the $n$-derived category of $\cM(
\Lambda )$.

\begin{Theorem}
[=Theorem \ref{thm:intermediate}]
\label{thm:C}
There is a bijection
\[
  \biggl\{
    \cU \subseteq \cM( \Lambda )
  \,\bigg|\,
    \cU \mbox{ is a torsion class}
  \biggr\}
  \leftrightarrow
  \biggl\{
    \cX \subseteq \cC( \Lambda )
  \,\bigg|
    \begin{array}{ll}
      \mbox{$\cX$ is an aisle with} \\
      \cC^{\leqslant -n}( \Lambda ) \subseteq \cX \subseteq \cC^{\leqslant 0}( \Lambda ) \\
    \end{array}
  \biggr\}.
\]
\end{Theorem}

We refer the reader to Definition \ref{def:C_leq} for the notation
$\cC^{\leqslant 0}( \Lambda )$, but state our definition of aisles.

\begin{Definition}
[Aisles]
\label{def:t-structures}
An {\em aisle} in an $( n+2 )$-angulated category is a
torsion class $\cX$ satisfying $\Sigma^n \cX \subseteq \cX$.
\hfill $\Box$
\end{Definition}

It is reasonable to ask if there are versions of Theorems \ref{thm:B}
and \ref{thm:C} for more general $n$-abelian and $( n+2 )$-angulated
categories.  As far as Theorem \ref{thm:B} is concerned, the current
version of Property $(E)$ refers to an ambient triangulated category
of $\cM( \Lambda )$, and it is not yet clear how to get an equivalent
``internal'' property which makes sense for general $n$-abelian
categories.  Theorem \ref{thm:C} appears to face a more serious
obstruction, since we do not presently know how to pass from a general
$n$-abelian category to an $n$-derived $( n+2 )$-angulated category.
Indeed, the only setting where we know how to do so is that
of $\cM( \Lambda )$ and $\cC( \Lambda )$.  This needs to be
generalised before one can hope to generalise Theorem \ref{thm:C}.
Another question is whether an aisle in an $( n+2 )$-angulated category
gives rise to an $n$-abelian heart.  There are special cases where the
answer appears to be yes; indeed, the equation $\cC^{\leqslant 0}(
\Lambda ) \cap \cC^{\geqslant 0}( \Lambda ) = \cM( \Lambda )$ suggests
that the heart corresponding to $\cC^{\leqslant 0}( \Lambda )$ is
$\cM( \Lambda )$.  However, we do not yet have a construction which
associates a heart to a general aisle.

To round off, note that everything we have said can be dualised,
resulting in torsion free classes and co-aisles.

The paper is organised as follows.  Section \ref{sec:lemmas} gives
some background.  Section \ref{sec:Wakamatsu} gives an $( n+2
)$-angulated version of Wakamatsu's Lemma.  Section
\ref{sec:categories} prepares the rest of the paper by giving some
background on the categories $\cM( \Lambda )$ and $\cC( \Lambda )$.
Section \ref{sec:E} introduces Property $(E)$ and shows some initial
results about it.  Sections \ref{sec:n-abelian} through
\ref{sec:example1} prove Theorems \ref{thm:A} through \ref{thm:C},
though not in that order.  Section \ref{sec:example2} shows an example
of a non-splitting torsion class in $\cM( \Lambda )$.

Our terminology is mainly standard.  We use the words (pre)cover
and (pre)envelope due to Enochs \cite[sec.\ 1]{E}.  They translate as
follows to Auslander and Smal\o's terminology \cite{AS}: A precover is
a right approximation, a cover is a minimal right approximation, a
preenvelope is a left approximation, and an envelope is a minimal left
approximation.

\section{Background}
\label{sec:lemmas}

Recall that throughout, $n \geqslant 1$ is a fixed integer.  We start
with the definition of $n$-abelian categories.

\begin{Definition}
[Jasso {\cite[defs.\ 2.2, 2.4, and 3.1]{J}}]
\label{def:n-abelian}
Let $\cM$ be an additive category.

An {\em $n$-kernel} of a morphism $m_1 \stackrel{ \mu_1 }{ \rightarrow
} m_0$ is a sequence $m_{ n+1 } \stackrel{ \mu_{ n+1 } }{
  \longrightarrow } \cdots \stackrel{ \mu_2 }{ \longrightarrow } m_1$
such that
\begin{equation}
\label{equ:n-kernel}
  0
  \longrightarrow \cM( m,m_{ n+1 } )
  \stackrel{ ( \mu_{ n+1 } )_* }{ \longrightarrow } \cdots
  \stackrel{ ( \mu_{ 2 } )_* }{ \longrightarrow } \cM( m,m_1 )
  \stackrel{ ( \mu_{ 1 } )_* }{ \longrightarrow } \cM( m,m_0 )
\end{equation}
is exact for each $m \in \cM$.  The notion of {\em $n$-cokernel} is
defined dually.

An {\em $n$-exact sequence} is a sequence $m_{ n+1 }
\stackrel{ \mu_{ n+1 } }{ \longrightarrow } \cdots \stackrel{ \mu_1 }{
  \longrightarrow } m_0$ such that
\eqref{equ:n-kernel} and
\[
  0
  \longrightarrow \cM( m_{ 0 },m )
  \stackrel{ \mu_{ 1 }^{ * } }{ \longrightarrow } \cdots
  \stackrel{ \mu_{ n }^{ * } }{ \longrightarrow } \cM( m_n,m )
  \stackrel{ \mu_{ n+1 }^{ * } }{ \longrightarrow } \cM( m_{ n+1 },m )
\]
are exact for each $m \in \cM$.  That is, the morphisms $\mu_{ n+1 },
\ldots, \mu_2$ form an $n$-kernel of $\mu_1$ and the morphisms $\mu_n,
\ldots, \mu_1$ form an $n$-cokernel of $\mu_{ n+1 }$.  In particular,
$\mu_{ n+1 }$ is monic and $\mu_1$ is epic and the sequence is often
written
\begin{equation}
\label{equ:n-exact_sequence}
  0 \longrightarrow m_{ n+1 } \stackrel{ \mu_{ n+1 } }{
  \longrightarrow } \cdots \stackrel{ \mu_1 }{ \longrightarrow } m_0  
  \longrightarrow 0.
\end{equation}

We say that $\cM$ is {\em $n$-abelian} if it satisfies the following. 
\begin{itemize}[leftmargin=70pt,itemsep=3pt]

  \item[{[$n$AB0]}]  $\cM$ has split idempotents.

  \item[{[$n$AB1]}]  Each morphism has an $n$-kernel and an
    $n$-cokernel. 

  \item[{[$n$AB2]}] If $m_{ n+1 } \stackrel{ \mu_{ n+1 } }{
      \rightarrow } m_n$ is monic and $m_n \stackrel{ \mu_n }{
      \longrightarrow } \cdots \stackrel{ \mu_1 }{ \longrightarrow }
    m_0$ is an $n$-cokernel of $\mu_{ n+1 }$, then
    \eqref{equ:n-exact_sequence} is an $n$-exact sequence.

  \item[{[$n$AB2${}^{\opp}$]}] If $m_1 \stackrel{ \mu_1 }{ \rightarrow
    } m_0$ is epic and $m_{ n+1 } \stackrel{ \mu_{ n+1 } }{
      \longrightarrow } \cdots \stackrel{ \mu_2 }{ \longrightarrow }
    m_1$ is an $n$-kernel of $\mu_1$, then
    \eqref{equ:n-exact_sequence} is an $n$-exact sequence.
    \hfill $\Box$

\end{itemize}
\end{Definition}

The following is an abridged form of the definition of
$(n+2)$-angulated categories.

\begin{Definition}
[{Geiss, Keller, and Oppermann \cite[def.\ 2.1]{GKO}}]
\label{def:n-angulated}
Let $\cC$ be an additive category, let $\Sigma^n$ an automorphism of
$\cC$, referred to as the $n$-suspension functor (this is just
notation; there is no assumption that $\Sigma^n$ is in fact an $n^{\rm
  th}$ power), and let $\pentagon$ be a class of diagrams of the form
\[
  c^0 \rightarrow c^1 \rightarrow \cdots \rightarrow c^n \rightarrow
  c^{ n+1 } \rightarrow \Sigma^n c^0
\]
called {\em $( n+2 )$-angles}.  The triple $( \cC , \Sigma^n ,
\pentagon )$ is called an {\em $( n+2 )$-angulated category} if it
satisfies axioms (F1)--(F4) of \cite[def.\ 2.1]{GKO}.

Among the axioms, we highlight that for each $c \in \cC$, the diagram
$c \stackrel{ \id_c }{ \rightarrow } c \rightarrow 0 \rightarrow
\cdots \rightarrow 0 \rightarrow \Sigma^n c$ is an $( n+2 )$-angle.
Another axiom concerns diagrams of the following form.
\[
  \xymatrix {
    c^0 \ar[r] \ar_{ \gamma }[d] & c^1 \ar[r] \ar[d]
      & c^2 \ar[r] \ar@{.>}[d] & \cdots \ar[r]
      & c^{ n+1 } \ar[r] \ar@{.>}[d] 
      & \Sigma^n c^0 \ar^{ \Sigma^n \gamma }[d] \\
    \widetilde{ c }^0 \ar[r] & \widetilde{ c }^1 \ar[r] 
      & \widetilde{ c }^2 \ar[r] & \cdots \ar[r]
      & \widetilde{ c }^{ n+1 } \ar[r]
      & \Sigma^n \widetilde{ c }^0\\
            }
\]
If the rows are $( n+2 )$-angles and the commutative square on the left
is given, then there exist stippled arrows completing to a commutative
diagram.
\hfill $\Box$
\end{Definition}

Let us recall that $n$-abelian and $( n+2 )$-angulated categories were
introduced to formalise the behaviour of $n$-cluster tilting
subcategories.

\begin{Definition}
[{Iyama \cite[def.\ 1.1]{I1}, Keller and Reiten \cite[sec.\ 5.1]{KR}}]
\label{def:cluster_tilting}
A full subcategory $\cT$ of an abelian or triangulated category $\cD$ is
called {\em $n$-cluster tilting} if it is functorially finite,
satisfies 
\begin{align*}
  \cT
  & = \{\, d \in \cD \,|\, \Ext^1( \cT,d ) = \cdots = \Ext^{ n-1 }( \cT,d ) = 0 \,\} \\ 
  & = \{\, d \in \cD \,|\, \Ext^1( d,\cT ) = \cdots = \Ext^{ n-1 }( d,\cT ) = 0 \,\},
\end{align*}
and, in the abelian case, is generating and co-generating, that is,
each object $d \in \cD$ permits an epic and a monic $t
\twoheadrightarrow d \hookrightarrow t'$ with $t,t' \in \cT$.

An object $t$ is called {\em $n$-cluster tilting} if $\add( t )$ is an
$n$-cluster tilting subcategory.
\hfill $\Box$
\end{Definition}

\begin{Theorem}
[{Jasso \cite[thm.\ 3.16]{J}}]
\label{thm:J}
An $n$-cluster tilting subcategory of an abelian category is
$n$-abelian. 
\end{Theorem}

\begin{Theorem}
[{Geiss, Keller, and Oppermann \cite[thm.\ 1]{GKO}}]
\label{thm:GKO}
Let $\cD$ be a triangulated category with suspension functor $\Sigma$.
An $n$-cluster tilting subcategory $\cT \subseteq \cD$ satisfying
$\Sigma^n \cT \subseteq \cT$ can be equipped with a structure of $(
n+2 )$-angulated category in which the $n$-suspension is the $n^{\rm th}$
power $\Sigma^n$.
\end{Theorem}

We leave the statement of Theorem \ref{thm:GKO} vague by not
explaining how the $( n+2 )$-angles are defined.  The reader is
referred to \cite[thm.\ 1]{GKO} for the details.

To round off, two lemmas which will be needed later.

\begin{Lemma}
[{Geiss, Keller, and Oppermann \cite[def.\ 2.1 and prop.\ 2.5(a)]{GKO}}]
\label{lem:exact2}
Let $\cC$ be an $( n+2 )$-angulated category with an $( n+2 )$-angle
\[
  c^0
  \rightarrow \cdots
  \rightarrow c^{ n+1 }
  \rightarrow \Sigma^n c^0.
\]
For each $c \in \cC$, the induced sequence
\[
  \cC( c,c^0 )
  \rightarrow \cdots
  \rightarrow \cC( c,c^{ n+1 } )
  \rightarrow \cC( c,\Sigma^n c^0 )
\]
is exact.
\end{Lemma}

\begin{Lemma}
\label{lem:precovers}
\begin{enumerate}

  \item  Let $\cM$ be an $n$-abelian category, $\cU \subseteq \cM$ a
    full subcategory.  Consider an $n$-exact sequence
$
  0
  \rightarrow u 
  \stackrel{ \theta }{ \rightarrow } m
  \stackrel{ \mu }{ \rightarrow } v^1
  \rightarrow \cdots
  \stackrel{ \psi }{ \rightarrow } v^n
  \rightarrow 0
$
in $\cM$.

\smallskip

\begin{enumerate}

  \item  For $n = 1$, the sequence satisfies the conditions
in Definition \ref{def:torsion} if and only if $\theta$ is a
$\cU$-precover and $\cM( \cU,v^1 ) = 0$.

\smallskip

  \item For $n \geqslant 2$, the sequence satisfies the conditions in
  Definition \ref{def:torsion} if and only if $\theta$ is a
  $\cU$-precover and the induced map $\cM( \widetilde{ u },\psi )$ is
  surjective for each $\widetilde{ u } \in \cU$.

\end{enumerate}

\smallskip

In either case, ``$\cU$-precover'' can be replaced with
``$\cU$-cover''.

\smallskip

  \item  Let $\cC$ be an $( n+2 )$-angulated category, $\cX \subseteq
    \cC$ a full subcategory.  An $( n+2 )$-angle
$
  x 
  \stackrel{ \xi }{ \rightarrow } c
  \rightarrow y^1
  \rightarrow \cdots
  \rightarrow y^n
  \stackrel{ \varphi }{ \rightarrow } \Sigma^n x
$
in $\cC$ satisfies the conditions in Definition \ref{def:torsion} if
and only if $\xi$ is an $\cX$-precover and the induced map $\cC(
\widetilde{ x },\varphi )$ is zero for each $\widetilde{ x } \in \cX$.

\smallskip

  \item  A torsion class in an $n$-abelian category is closed under
    sums and summands.

\end{enumerate}
\end{Lemma}

\begin{proof}
(i)  Let $n \geqslant 2$ and $\widetilde{ u } \in \cU$.  The sequence
\[
  0
  \rightarrow \cM( \widetilde{ u },u )
  \stackrel{ \theta_*}{ \rightarrow } \cM( \widetilde{ u },m )
  \stackrel{ \mu_* }{ \rightarrow } \cM( \widetilde{ u },v^1 )
  \rightarrow \cdots
  \stackrel{ \psi_*}{ \rightarrow } \cM( \widetilde{ u },v^n )
\]
is exact by Definition \ref{def:n-abelian} (indeed, this is even true
for $\widetilde{ u } \in \cM$).  Hence the conditions in Definition
\ref{def:torsion} amount to $\mu_*$ being zero and $\psi_*$ being
surjective for each $\widetilde{ u } \in \cU$.  This is equivalent to
$\theta_*$ and $\psi_*$ being surjective for each $\widetilde{ u } \in
\cU$, that is, to $\theta$ being a $\cU$-precover and $\psi_* = \cM(
\widetilde{ u },\psi )$ being surjective for each $\widetilde{ u } \in
\cU$.

The case $n = 1$ is handled similarly.  Finally, $\theta$ is monic so
it is a $\cU$-precover if and only if it is a $\cU$-cover.

(ii) This is proved in the same style as (i), using Lemma
\ref{lem:exact2}.  Note that $\xi$ is not in general monic, so even if
it is an $\cX$-precover it may not be an $\cX$-cover.

(iii)  Consider the $n$-exact sequence from Definition
\ref{def:torsion} and let $\widetilde{ u } \in \add( \cU )$ be given.
The induced map $\cM( \widetilde{ u },u ) \stackrel{ \theta_*}{
  \rightarrow } \cM( \widetilde{ u },m )$ is injective since $\theta$
is monic, and the proof of part (i) still shows that $\theta_*$ is
surjective.  So $\theta_*$ is bijective whence the restrictions of the
functors $\cM( -,u )$ and $\cM( -,m )$ to $\add( \cU )$ are
equivalent.  In particular, if $m \in \add( \cU )$ then $m \cong u \in
\cU$.
\end{proof}

\section{Wakamatsu's Lemma for $( n+2 )$-angulated categories}
\label{sec:Wakamatsu}

Recall that throughout, $n \geqslant 1$ is a fixed integer.  The
notion of being left closed under $n$-extensions was defined after
Theorem \ref{thm:n-angulated_torsion2} which is itself an immediate
consequence of the following result.

\begin{Lemma}
[$( n+2 )$-Angulated Wakamatsu's Lemma]
\label{lem:Wakamatsu}
Let $\cC$ be an $( n+2 )$-angulated category, $\cX \subseteq \cC$ a
full subcategory left closed under $n$-extensions.  If $x \stackrel{
  \xi }{ \rightarrow } c$ is an $\cX$-cover, then, in each
completion to an $( n+2 )$-angle
\[
  x \stackrel{ \xi }{ \rightarrow } c
  \rightarrow y^1 \rightarrow \cdots \rightarrow y^n
  \stackrel{ \varphi }{ \rightarrow } \Sigma^n x,
\]
we have $y^1 \rightarrow \cdots \rightarrow y^n$ in $\Xexact$.
\end{Lemma}

\begin{proof}
By Lemma \ref{lem:precovers}(ii) we only need to show
$\cC( \widetilde{x},\varphi ) = 0$ for $\widetilde{ x } \in \cX$.

Let $\widetilde{x} \stackrel{ \widetilde{\xi} }{ \rightarrow } y^n$ be
given.  We must show $\varphi\widetilde{\xi} = 0$.  Using axioms
(F1)(c), (F2), and (F3) of \cite[def.\ 2.1]{GKO} gives a
commutative diagram where the first row is a completion of
$\varphi\widetilde{\xi}$ to an $( n+2 )$-angle.
\[
  \xymatrix {
    x \ar^{\psi}[r] \ar@{=}[d] & c^0 \ar[r] \ar^{\gamma}[d]
      & c^1 \ar[r] \ar[d] & \cdots \ar[r] & c^{ n-1 } \ar[r] \ar[d]
      & \widetilde{x} \ar^{\varphi\widetilde{\xi}}[r] \ar_{\widetilde{\xi}}[d] 
      & \Sigma^n x \ar@{=}[d] \\
    x \ar_{\xi}[r] & c \ar[r] & y^1 \ar[r] & \cdots \ar[r]
      & y^{ n-1 } \ar[r] & y^n \ar_{\varphi}[r] & \Sigma^n x
            }
\]
Since $\cX$ is left closed under $n$-extensions, we can assume $c^0
\in \cX$ whence there exists $c^0 \stackrel{ \theta }{ \rightarrow }
x$ such that $\xi\theta = \gamma$.  Hence $\xi \circ \theta\psi =
\gamma\psi = \xi$ whence $\theta\psi$ is an isomorphism since $\xi$
is a cover.

It follows from Lemma \ref{lem:exact2} that the composition
of two consecutive morphisms in an $( n+2 )$-angle is zero, so $\psi
\circ \Sigma^{-n}( \varphi\widetilde{\xi} ) = 0$ and hence $\theta\psi
\circ \Sigma^{-n}( \varphi\widetilde{\xi} ) = 0$.  Since $\theta\psi$
is an isomorphism this shows $\Sigma^{-n}( \varphi\widetilde{\xi} ) =
0$ whence $\varphi\widetilde{\xi} = 0$ as desired.
\end{proof}

\section{$n$-representation finite algebras and the associated
  $n$-abelian and $( n+2 )$-angulated categories}
\label{sec:categories}

Recall that throughout, $n \geqslant 1$ is a fixed integer.

\begin{Setup}
\label{set:blanket}
In the rest of the paper, $k$ is an algebraically closed field,
$\Lambda$ is a fixed finite dimensional $k$-algebra, $\mod( \Lambda )$
is the category of finitely generated left $\Lambda$-modules, and
$\Db( \mod\,\Lambda )$ is the bounded derived category.

We assume that $\Lambda$ is weakly $n$-representation finite, that is,
there is an $n$-cluster tilting object $M \in \mod( \Lambda )$, see
\cite[def.\ 2.2]{IO} or Definition \ref{def:cluster_tilting}.  Such an
$M$ is known as an $n$-cluster tilting left $\Lambda$-module.  We also
assume that $\gldim( \Lambda ) \leqslant n$ and that $\Lambda$ (viewed
as a left $\Lambda$-module) and $M$ are basic, that is, without
repeated direct summands.  Note that $M$ is unique by \cite[thm.\
1.6]{I1}.

Let $\cM( \Lambda ) = \add(M)$ be the $n$-cluster tilting subcategory
of $\mod( \Lambda )$ corresponding to $M$, and let $\cC( \Lambda ) =
\add\,\{\, \Sigma^{in}M \,|\, i \in \BZ \,\}$ be the $n$-cluster
tilting subcategory of $\Db( \mod\,\Lambda )$ constructed in
\cite[thm.\ 1.21]{I1}; see \cite[def.\ 1.1]{I1} and \cite[sec.\
5.1]{KR} or Definition \ref{def:cluster_tilting}.

Since $\Lambda$ is fixed, we henceforth write $\cM$ and $\cC$ instead
of $\cM( \Lambda )$ and $\cC( \Lambda )$.  The category $\cM$ is
$n$-abelian by \cite[thm.\ 3.16]{J}, see Theorem \ref{thm:J}, and
$\cC$ is $( n+2 )$-angulated by \cite[thm.\ 1]{GKO}, see Theorem
\ref{thm:GKO}.  As remarked in the introduction, it seems reasonable
to think of $\cC$ as the $n$-derived category of $\cM$, although we
presently lack the definitions to give this sentence a precise
meaning.  \hfill $\Box$
\end{Setup}

\begin{Lemma}
\label{lem:functorially_finite}
\begin{enumerate}

  \item  Each full subcategory $\cU \subseteq \cM$ closed under sums and
  summands is functorially finite in each of $\cM$, $\mod( \Lambda )$,
  and $\Db( \mod\,\Lambda )$.

\smallskip

  \item  Each full subcategory $\cX \subseteq \cC$ closed under sums
    and summands is functorially finite in each of $\cC$ and
    $\Db( \mod\,\Lambda )$.

\end{enumerate}
\end{Lemma}

\begin{proof}
(i)  Immediate since $\cM = \add( M )$ has only finitely many
indecomposable objects. 

(ii)  Since $\cM$ has only finitely many indecomposable objects and
since $\Lambda$ has finite global dimension, each object in $\Db(
\mod\,\Lambda )$ has non-zero morphisms only to and only from finitely
many indecomposable objects in $\cC$.  This implies the claim.
\end{proof}

\section{Property $(E)$}
\label{sec:E}

Recall that throughout, $n \geqslant 1$ is a fixed integer and we are
working under Setup \ref{set:blanket}.  We think of the following as
an $n$-homological version of the property that $\cU \subseteq \cM$ is
closed under quotients and extensions.  Indeed, it is equivalent to
this for $n = 1$ under mild assumptions as we will show in Proposition
\ref{pro:1}.

\begin{Definition}
[Property $(E)$]
\label{def:E}
Let $\cU \subseteq \cM$ be a full subcategory.  Consider a triangle 
\[
  \Sigma^{ -n }u'' \oplus m''
  \rightarrow u'
  \rightarrow e
  \rightarrow \Sigma^{ -n+1 }u'' \oplus \Sigma m''
\]
in $\Db( \mod\,\Lambda )$ where $u',u'' \in \cU$ and $m'' \in \cM$.

If, for each such triangle, the object $e \in \Db( \mod\,\Lambda )$ has
an $\cM$-preenvelope $e \rightarrow u$ with $u \in \cU$, then we say
that $\cU$ {\em has Property $(E)$}.
\hfill $\Box$
\end{Definition}

The main purpose of this section is to show that Property $(E)$ is
equivalent to the following which is sometimes handier to use.

\begin{Definition}
[Property $(F)$]
\label{def:F}
Let $\cU \subseteq \cM$ be a full subcategory.  Consider a triangle
\[
  \Sigma^{ -n }u''
  \rightarrow b'
  \rightarrow f
  \rightarrow \Sigma^{ -n+1 }u''
\]
in $\Db( \mod\,\Lambda )$ where $u'' \in \cU$ and where $b' \in \mod(
\Lambda )$ is a quotient in $\mod( \Lambda )$ of a module $u' \in
\cU$. 

If, for each such triangle, the object $f \in \Db( \mod\,\Lambda )$ has
an $\cM$-preenvelope $f \rightarrow u$ with $u \in \cU$, then we say
that $\cU$ {\em has Property $(F)$}.
\hfill $\Box$
\end{Definition}

\begin{Remark}
\label{rmk:truncation_triangle}
Recall that each object $d \in \Db( \mod\,\Lambda )$ has a standard
truncation triangle
\[
  \tau^{ \leqslant -1 }d \rightarrow d \rightarrow \tau^{ \geqslant 0 }d
  \rightarrow \Sigma \tau^{ \leqslant -1 }d
\]
which is determined up to isomorphism by the properties that $\tau^{
  \leqslant -1 }d$ has cohomology concentrated in degrees $\leqslant
-1$ and $\tau^{ \geqslant 0 }d$ has cohomology concentrated in degrees
$\geqslant 0$.  \hfill $\Box$
\end{Remark}

\begin{Lemma}
\label{lem:2}
Let $d \in \Db( \mod\,\Lambda )$ and $m \in \cM$ be fixed objects.
There is an $\cM$-preenvelope of the form $d \rightarrow m$ if and
only if there is an $\cM$-preenvelope of the form $\tau^{ \geqslant 0 }d
\rightarrow m$.
\end{Lemma}

\begin{proof}
For $\widetilde{m} \in \cM$, the triangle from Remark
\eqref{rmk:truncation_triangle} induces an exact sequence
\[
  \Hom_{ \Db }( \Sigma\tau^{ \leqslant -1 }d,\widetilde{m} )
  \rightarrow \Hom_{ \Db }( \tau^{ \geqslant 0 }d,\widetilde{m} )
  \rightarrow \Hom_{ \Db }( d,\widetilde{m} )
  \rightarrow \Hom_{ \Db }( \tau^{ \leqslant -1 }d,\widetilde{m} ).
\]
The cohomology of $\tau^{ \leqslant -1 }d$ is concentrated in degrees $\leqslant
-1$ while the cohomology of $\widetilde{m}$ is in degree $0$, so the
two outer terms in the sequence are $0$.  This implies the lemma.
\end{proof}

\begin{Lemma}
\label{lem:e_tau_f}
Let
\[
  \xymatrix @+0.5pc {
    \Sigma^{ -n }u'' \oplus m'' \ar^-{( \varphi'',\mu'')}[r]
    & u' \ar[r]
    & e \ar[r] 
    & \Sigma^{ -n+1 }u'' \oplus \Sigma m''
           }
\]
be a triangle in $\Db( \mod\,\Lambda )$ with $m'', u',u'' \in \mod(
\Lambda )$.  Set $b' = \Coker \mu''$ and consider the right exact
sequence
\[
  m''
  \stackrel{ \mu'' }{ \longrightarrow } u'
  \stackrel{ \mu' }{ \longrightarrow } b'
  \longrightarrow 0.
\]

Then each triangle of the form
\[
  \Sigma^{ -n }u''
  \stackrel{ \mu'\varphi'' }{ \longrightarrow } b'
  \longrightarrow f
  \longrightarrow \Sigma^{ -n+1 }u''
\]
satisfies $f \cong \tau^{ \geqslant 0 }e$.
\end{Lemma}

\begin{proof}
The octahedral axiom gives a commutative diagram where the rows and
columns are triangles minus the fourth object, and the first vertical
triangle is split.
\[
  \xymatrix{
    m'' \ar@{=}[r] \ar[d] & m'' \ar[r] \ar^{\mu''}[d] & 0 \ar[d] \\
    \Sigma^{ -n }u'' \oplus m'' \ar_-{( \varphi'',\mu'')}[r] \ar[d] & u' \ar[r] \ar^{\pi'}[d] & e \ar@{=}[d]\\
    \Sigma^{ -n }u'' \ar_-{\pi'\varphi''}[r] & q' \ar[r] & e
           }
\]
The long exact cohomology sequence of the second vertical triangle
shows that the only non-zero cohomology of $q'$ is
\begin{align*}
  \H^{-1}(q') & = \Ker \mu'' =: a', \\
  \H^0(q')   & = \Coker \mu'' = b',
\end{align*}
so there is a standard truncation triangle $\Sigma a' \longrightarrow
q' \stackrel{ \rho' }{ \longrightarrow } b'$ where we observe that
$\rho'\pi' = \mu'$.  The octahedral axiom gives a commutative diagram
where the rows and columns are triangles minus the fourth object.
\[
  \xymatrix{
    0 \ar[r] \ar[d] & \Sigma^{ -n }u'' \ar@{=}[r] \ar_{\pi'\varphi''}[d] & \Sigma^{ -n }u'' \ar^{\mu'\varphi''}[d] \\
    \Sigma a' \ar[r] \ar@{=}[d] & q' \ar_{\rho'}[r] \ar[d] & b' \ar[d] \\
    \Sigma a' \ar[r] & e \ar[r] & f
           }
\]

Here the third horizontal triangle is isomorphic to the standard
truncation triangle $\tau^{ \leqslant -1 }e \rightarrow e \rightarrow
\tau^{ \geqslant 0 }e$, proving the lemma.  This holds because the
cohomology of $\Sigma a'$ is concentrated in degrees $\leqslant -1$
and the cohomology of $f$ is concentrated in degrees $\geqslant 0$;
these properties determine the standard truncation triangle up to
isomorphism by Remark \ref{rmk:truncation_triangle}.  To verify the
statement about the cohomology of $f$, use the long exact cohomology
sequence of the third vertical triangle.
\end{proof}

\begin{Lemma}
\label{lem:EeqF}
Let $\cU \subseteq \cM$ be a full subcategory.  Then $\cU$ has
Property $(E)$ $\Leftrightarrow$ $\cU$ has Property $(F)$.
\end{Lemma}

\begin{proof}
``$\Rightarrow$'': Consider the triangle from Definition \ref{def:F},
\begin{equation}
\label{equ:E_triangle}
  \Sigma^{ -n }u''
  \stackrel{ \psi'' }{ \longrightarrow } b'
  \longrightarrow f
  \longrightarrow \Sigma^{ -n+1 }u'',
\end{equation}
with $u'' \in \cU$ and $b'$ a quotient of $u' \in \cU$.  We must show
there is an $\cM$-preenvelope $f \rightarrow u$ with $u \in \cU$.

Pick a right exact sequence
\[
  m''
  \stackrel{ \mu'' }{ \longrightarrow } u'
  \stackrel{ \mu' }{ \longrightarrow } b'
  \longrightarrow 0
\]
in $\mod( \Lambda )$ with $m''$ projective and note that $m'' \in \cM$
since $\cM$ is $n$-cluster tilting, see \cite[def.\ 1.1]{I1} or
Definition \ref{def:cluster_tilting}.  Set $k' = \Ker \mu'$.  There is
a short exact sequence $0 \rightarrow k' \rightarrow u' \stackrel{
  \mu' }{ \rightarrow } b' \rightarrow 0$ in $\mod( \Lambda )$ which
induces a triangle $k' \rightarrow u' \stackrel{ \mu' }{ \rightarrow }
b' \rightarrow \Sigma k'$ in $\Db( \mod\,\Lambda )$.  Since $u''$ and
$k'$ are in $\mod( \Lambda )$ we have
$
  \Hom_{ \Db }( \Sigma^{ -n }u'',\Sigma k' )
  \cong \Ext_{ \Lambda }^{ n+1 }( u'',k' ) = 0
$
because $\gldim \Lambda \leqslant n$, so $\Sigma^{ -n }u'' \stackrel{
  \psi'' }{ \longrightarrow } b'$ factors as
\[
  \Sigma^{ -n }u''
  \stackrel{ \varphi'' }{ \longrightarrow } u'
  \stackrel{ \mu' }{ \longrightarrow } b'.
\]

We have constructed morphisms $\mu''$ and $\varphi''$ which can be
used in the first triangle of Lemma \ref{lem:e_tau_f}.  Since
$\mu'\varphi'' = \psi''$, the second triangle of Lemma
\ref{lem:e_tau_f} can be taken to be \eqref{equ:E_triangle} so $f
\cong \tau^{ \geqslant 0 }e$.  By Property $(E)$ there is an
$\cM$-preenvelope $e \rightarrow u$ with $u \in \cU$, and by Lemma
\ref{lem:2} there is hence an $\cM$-preenvelope $f \cong \tau^{
  \geqslant 0 } e \rightarrow u$ with $u \in \cU$ as desired.

``$\Leftarrow$'': Consider the triangle from Definition \ref{def:E}.
To show that $\cU$ has Property $(E)$, it is enough by Lemma
\ref{lem:2} to show that there is an $\cM$-preenvelope $\tau^{
  \geqslant 0 }e \rightarrow u$ with $u \in \cU$.  This holds by Lemma
\ref{lem:e_tau_f} and Property $(F)$.
\end{proof}

For the following result, note that $n = 1$ implies that $\Lambda$ is
a hereditary algebra of finite representation type and that $\cM =
\mod( \Lambda )$.

\begin{Proposition}
\label{pro:1}
Assume $n = 1$ and let $\cU \subseteq \cM$ be a full subcategory
closed under summands.  Then $\cU$ has Property $(E)$ if and only if
it is closed under quotients and extensions.
\end{Proposition}

\begin{proof}
By Lemma \ref{lem:EeqF}, we can replace Property $(E)$ by Property
$(F)$ in the proposition.  The triangle in Definition \ref{def:F} has
the form
$
  \Sigma^{ -1 }u''
  \rightarrow b'
  \rightarrow f
  \rightarrow u'',
$
so $f$ is an extension in $\Db( \mod\,\Lambda )$ of the modules $b'$
and $u''$, whence $f \in \mod( \Lambda )$.  Since $\cM = \mod( \Lambda
)$, an $\cM$-preenvelope $f \rightarrow u$ must be split injective, so
since $\cU$ is closed under summands, the preenvelope can be chosen
with $u \in \cU$ if and only if $f \in \cU$.

It follows that $\cU$ has Property $(F)$ if and only if $f \in \cU$
for each extension $0 \rightarrow b' \rightarrow f \rightarrow u''
\rightarrow 0$ in $\mod( \Lambda )$ with $u'' \in \cU$ and $b'$ a
quotient in $\mod( \Lambda )$ of a module $u' \in \cU$.  This proves
the proposition.
\end{proof}

\section{Torsion classes in $n$-abelian categories associated to
  $n$-representation finite algebras}
\label{sec:n-abelian}

Recall that throughout, $n \geqslant 1$ is a fixed integer and we are
working under Setup \ref{set:blanket}.  This section proves Theorem
\ref{thm:B} from the introduction, see Theorem
\ref{thm:n-abelian_torsion}. 

\begin{Lemma}
\label{lem:exact}
An $n$-exact sequence 
$
  0 \rightarrow m^0 \rightarrow \cdots \rightarrow m^{n+1} \rightarrow 0
$
in $\cM$ is an exact sequence in $\mod( \Lambda )$. 
\end{Lemma}

\begin{proof}
The projective generator $\Lambda$ of $\mod( \Lambda )$ is in the
$n$-cluster tilting subcategory $\cM$.  Applying the functor $\cM(
\Lambda , - )$ to the $n$-exact sequence and applying Definition
\ref{def:n-abelian} shows that $0 \rightarrow m^0 \rightarrow \cdots
\rightarrow m^{ n+1 }$ is exact.  A dual argument shows that $m^0
\rightarrow \cdots \rightarrow m^{ n+1 } \rightarrow 0$ is exact.
\end{proof}

\begin{Lemma}
\label{lem:Ext_chain}
Assume $n \geqslant 2$ and let 
\[
  v^1
  \stackrel{ \psi^1 }{ \longrightarrow } v^2
  \stackrel{ \psi^2 }{ \longrightarrow } \cdots
  \stackrel{ \psi^{ n-2 } }{ \longrightarrow } v^{ n-1 }
  \stackrel{ \psi^{ n-1 } }{ \longrightarrow } v^n
\]
be an exact sequence in $\mod( \Lambda )$ with $v^1, \ldots, v^{ n-2 }
\in \cM$.  For each $m \in \cM$ there is an isomorphism
\[
  \Ext_{ \Lambda }^{ n-1 }( m,\Ker \psi^1 )
  \cong \Ext_{ \Lambda }^1( m,\Ker \psi^{ n-1 } ).
\]
\end{Lemma}

\begin{proof}
Break the exact sequence into short exact sequences, write down the
induced long exact sequences for $\Ext_{ \Lambda }^i( m,- )$, and use
that $\Ext_{ \Lambda }^i( \cM,\cM ) = 0$ for $i \in \{\, 1,\ldots,n-1
\,\}$ to get isomorphisms between all the $\Ext_{ \Lambda }^{ n-i }(
m,\Ker \psi^i )$. 
\end{proof}

\begin{Lemma}
\label{lem:vanishing}
Let $\cU \subseteq \cM$ be a full subcategory and let
\[
  0
  \longrightarrow u
  \stackrel{ \theta }{ \longrightarrow } m
  \longrightarrow v^1
  \stackrel{ \psi^1 }{ \longrightarrow } v^2
  \longrightarrow \cdots
  \longrightarrow v^{ n-1 }
  \stackrel{ \psi^{ n-1 } }{ \longrightarrow } v^n
  \longrightarrow 0
\]
be an $n$-exact sequence in $\cM$ which has $u \in \cU$ and 
$
  v^1
  \stackrel{ \psi^1 }{ \longrightarrow } v^2
  \longrightarrow \cdots
  \longrightarrow v^{ n-1 }
  \stackrel{ \psi^{ n-1 } }{ \longrightarrow } v^n
$
in $\Uexact$.  Write $c = \Coker
\theta$, where the cokernel is computed in $\mod( \Lambda )$.
Then
\begin{enumerate}

  \item  $\Hom_{ \Lambda }( \cU,c ) = 0$,

\medskip

  \item  $\Ext_{ \Lambda }^{ n-1 }( \cU,c ) = 0$.

\end{enumerate}
\end{Lemma}

\begin{proof}
(i)  Combine Lemma \ref{lem:exact} with the condition that 
$
  v^1
  \stackrel{ \psi^1 }{ \longrightarrow } v^2
  \longrightarrow \cdots
  \longrightarrow v^{ n-1 }
  \stackrel{ \psi^{ n-1 } }{ \longrightarrow } v^n
$
is in $\Uexact$.

(ii) For $n=1$ this coincides with (i) so assume $n \geqslant 2$.  The
$n$-exact sequence in the lemma is exact in $\mod( \Lambda )$ by Lemma
\ref{lem:exact}.  Let $\widetilde{u} \in \cU$ be given.  The short
exact sequence
$
  0
  \longrightarrow \Ker \psi^{ n-1 }
  \longrightarrow v^{ n-1 }
  \stackrel{ \psi^{ n-1 } }{ \longrightarrow } v^n
  \longrightarrow 0
$
gives an exact sequence 
\[
  \Hom_{ \Lambda }( \widetilde{u},v^{ n-1 } )
  \stackrel{ \psi^{ n-1 }_* }{ \longrightarrow }
  \Hom_{ \Lambda }( \widetilde{u},v^n )
  \longrightarrow
  \Ext_{ \Lambda }^1( \widetilde{u},\Ker \psi^{ n-1 } )
  \longrightarrow
  \Ext_{ \Lambda }^1( \widetilde{u},v^{ n-1 } ).
\]
Here $\psi^{ n-1 }_*$ is surjective because 
$
  v^1
  \stackrel{ \psi^1 }{ \longrightarrow } v^2
  \longrightarrow \cdots
  \longrightarrow v^{ n-1 }
  \stackrel{ \psi^{ n-1 } }{ \longrightarrow } v^n
$
is in $\Uexact$, and $\Ext_{ \Lambda }^1( \widetilde{u},v^{ n-1 } ) =
0$ because $\widetilde{u}$ and $v^{ n-1 }$ are in $\cM$, so $\Ext_{
  \Lambda }^1( \widetilde{u},\Ker \psi^{ n-1 } ) = 0$.  Combining with
the isomorphism in Lemma \ref{lem:Ext_chain} shows $\Ext_{ \Lambda }^{
  n-1 }( \widetilde{u},\Ker \psi^1 ) = 0$.  Finally, $\Ker \psi^1
\cong c$.
\end{proof}

\begin{Lemma}
\label{lem:injective}
Let $\cU \subseteq \cM$ be a full subcategory with Property $(E)$.
If $u \stackrel{ \theta }{ \rightarrow } m$ is a $\cU$-cover of $m
\in \cM$, then $\theta$ is a monomorphism in $\cM$.
\end{Lemma}

\begin{proof}
Lemma \ref{lem:EeqF} says that $\cU$ has Property $(F)$.  Let $u
\stackrel{ \theta }{ \rightarrow } m$ be a $\cU$-cover of $m \in \cM$,
and factorise it in $\mod( \Lambda )$ into a surjection followed by an
injection, $u \stackrel{ s }{ \twoheadrightarrow } b' \stackrel{ i }{
  \hookrightarrow } m$.  Property $(F)$ with $u'' = 0$ says that $b'$
has an $\cM$-preenvelope $b' \stackrel{ \beta' }{ \rightarrow } u'$
with $u' \in \cU$.  Factoring $i$ through $\beta'$ gives the
commutative diagram
\[
  \xymatrix {
    u \ar_{\theta}[dr] \ar@{->>}^{s}[r] & b' \ar^{\beta'}[r] \ar@{^{(}->}_<<<{i}[d] & u'\lefteqn{.} \ar^{\psi}[dl] \\
    & m & 
            }
\]
We can factorise $\psi$ through $\theta$, that is, $\psi =
\theta\alpha$.  But then
$\theta \circ \alpha\beta's
  = \psi\beta's
  = is
  = \theta$
whence $\alpha\beta's$ is an isomorphism since $\theta$ is a cover.
In particular, $\alpha\beta's$ is injective in $\mod( \Lambda )$.  It
follows that so are $s$ and $\theta = is$, and then $\theta$ is a
monomorphism in $\cM$.
\end{proof}

\begin{Theorem}
\label{thm:n-abelian_torsion}
Let $\cU \subseteq \cM$ be a full subcategory closed under sums and
summands.  Then
\[
  \mbox{ $\cU$ is a torsion class in $\cM$
         $\Leftrightarrow$
         $\cU$ has Property $(E)$. }
\]
\end{Theorem}

\begin{proof}
``$\Rightarrow$'': Consider the triangle from Definition \ref{def:E},
\[
  \Sigma^{ -n }u'' \oplus m''
  \rightarrow u'
  \rightarrow e
  \rightarrow \Sigma^{ -n+1 }u'',
\]
with $u',u'' \in \cU$ and $m'' \in \cM$.  By Lemma
\ref{lem:functorially_finite}(i) there is an $\cM$-preenvelope $e
\stackrel{ \varphi }{ \rightarrow } m$.

Since $\cU$ is a torsion class, there is an $n$-exact sequence in
$\cM$,
\[
  0
  \rightarrow u
  \stackrel{ \theta }{ \rightarrow } m
  \rightarrow v^1
  \rightarrow \cdots
  \rightarrow v^n
  \rightarrow 0,
\]
with $u \in \cU$ and $v^1 \rightarrow \cdots \rightarrow v^n$ in
$\Uexact$.  The sequence is exact in $\mod( \Lambda )$ by Lemma
\ref{lem:exact}.  Write $c = \Coker \theta$, where the cokernel is
computed in $\mod( \Lambda )$.  There is a triangle
\[
  \Sigma^{ -1 }c \rightarrow u 
  \stackrel{ \theta }{ \rightarrow } m \rightarrow c
\]
in $\Db( \mod\,\Lambda )$ and we can construct the following commutative
diagram. 
\[
  \xymatrix {
    \Sigma^{ -n }u'' \oplus m'' \ar[r] \ar_-{ \Sigma^{-1}\sigma }[d] & u' \ar[r] \ar_{\psi'}[d] & e \ar[r] \ar^{\varphi}[d] & \Sigma^{ -n+1 }u'' \oplus \Sigma m'' \ar^{ \sigma }[d] \\
    \Sigma^{ -1 }c \ar[r] & u \ar_{\theta}[r] & m \ar[r] & c
            }
\]
The morphism $\psi'$ exists because the composition $u' \rightarrow e
\stackrel{ \varphi }{ \rightarrow } m \rightarrow c$ is zero since
$\Hom_{\Lambda}( u',c ) = 0$ by Lemma \ref{lem:vanishing}(i).  The
morphism $\sigma$ is obtained by completing the diagram to a morphism
of triangles.

However, $\sigma = 0$ because $\Hom_{ \Db }( \Sigma^{ -n+1 }u'',c )
\cong \Ext_{ \Lambda }^{n-1}( u'',c ) = 0$ by Lemma
\ref{lem:vanishing}(ii) and $\Hom_{ \Db }( \Sigma m'',c ) = 0$ since
$m''$ and $c$ are in $\mod( \Lambda )$.  Hence $e \stackrel{ \varphi }{
  \rightarrow } m$ can be factored as $e \stackrel{ \varphi' }{
  \rightarrow } u \stackrel{ \theta }{ \rightarrow } m$, and then $e
\stackrel{ \varphi' }{ \rightarrow } u$ is an $\cM$-preenvelope with
$u \in \cU$, showing that $\cU$ has Property $(E)$.

``$\Leftarrow$'': Let $m \in \cM$ be given and let $u \stackrel{
  \theta }{ \rightarrow } m$ be a $\cU$-cover which exists by Lemma
\ref{lem:functorially_finite}(i).  By Lemma \ref{lem:injective} the
morphism $\theta$ is a monomorphism in $\cM$, so by \cite[def.\
3.1(A2)]{J} or Definition \ref{def:n-abelian}, item [$n$AB2], we can
complete it to the $n$-exact sequence
\[
  0
  \longrightarrow u
  \stackrel{ \theta }{ \longrightarrow } m
  \longrightarrow v^1
  \stackrel{ \psi^1 }{ \longrightarrow } v^2
  \longrightarrow \cdots
  \longrightarrow v^{ n-1 }
  \stackrel{ \psi^{ n-1 } }{ \longrightarrow } v^n
  \longrightarrow 0
\]
in $\cM$.  Set $c = \Coker \theta$ and let $u'' \in \cU$ be given.

For $n \geqslant 2$, by Lemma \ref{lem:precovers}(i) it remains to
show $\cM( u'',\psi^{ n-1 } )$ surjective.  The short exact sequence
$
  0
  \longrightarrow \Ker \psi^{ n-1 }
  \longrightarrow v^{ n-1 }
  \stackrel{ \psi^{ n-1 }}{ \longrightarrow } v^n
  \longrightarrow 0
$
means that it is sufficient to show $\Ext_{ \Lambda }^1( u'',\Ker
\psi^{ n-1 } ) = 0$, and combining Lemmas \ref{lem:exact} and
\ref{lem:Ext_chain} shows that this is equivalent to $\Ext_{ \Lambda
}^{ n-1 }( u'',\Ker \psi^1 ) = 0$.  Since $\Ker \psi^1 \cong c$ this
amounts to $\Hom_{ \Db }( \Sigma^{ -n+1 }u'',c ) = 0$.

For $n=1$, by Lemma \ref{lem:precovers}(i) it remains to show
$\Hom_{\Lambda}( u'',v^1 ) = 0$.  Since we have $v^1 \cong c$ in this
case, this again amounts to $\Hom_{ \Db }( \Sigma^{ -n+1 }u'',c ) =
0$.

So let $\Sigma^{ -n+1 }u'' \stackrel{ \sigma }{ \rightarrow } c$ be
given.  There is a triangle $\Sigma^{-1}c \rightarrow u \stackrel{
  \theta }{ \rightarrow } m \stackrel{ \mu }{ \rightarrow } c$ in
$\Db( \mod\, \Lambda )$, and we can spin it into a commutative diagram
where the rows are triangles, the first one continued by an additional
step.
\[
  \xymatrix {
    \Sigma^{ -n }u'' \ar^-{\Sigma^{-1}\delta}[r] \ar_-{ \Sigma^{-1}\sigma }[d] & u \ar[r] \ar@{=} [d] & e \ar^-{\alpha}[r] \ar^{\varphi}[d] & \Sigma^{ -n+1 }u'' \ar^{ \sigma }[d] \ar^-{\delta}[r] & \Sigma u\\
    \Sigma^{ -1 }c \ar[r] & u \ar_{\theta}[r] & m \ar_{\mu}[r] & c
            }
\]

Property $(E)$ says that there is an $\cM$-preenvelope $e \stackrel{
  \widetilde{ \varphi }}{ \rightarrow } \widetilde{ u }$ with
$\widetilde{ u } \in \cU$, and we can augment the diagram with
factorisations as follows where $\gamma$ exists because $\theta$ is a
$\cU$-cover. 
\[
  \xymatrix @-0.5pc{
    \Sigma^{ -n }u'' \ar^-{\Sigma^{-1}\delta}[rr]\ar[rr] \ar_-{ \Sigma^{-1}\sigma }[dd] && u \ar[rr] \ar@{=} [dd] && e \ar^-{\alpha}[rr] \ar^{\varphi}[dd] \ar_{\widetilde{\varphi}}[ld]&& \Sigma^{ -n+1 }u'' \ar^{ \sigma }[dd] \ar^-{\delta}[rr] && \Sigma u\\
    &&& \widetilde{u} \ar^<<<{\beta}[dr] \ar_<<<{\gamma}[dl] \\
    \Sigma^{ -1 }c \ar[rr] && u \ar_{\theta}[rr] && m \ar_{\mu}[rr] && c
            }
\]
Hence $\mu\varphi = \mu\beta\widetilde{\varphi} =
\mu\theta\gamma\widetilde{\varphi} = 0 \circ \gamma\widetilde{\varphi}
= 0$ whence $\sigma\alpha = 0$.  This implies that $\sigma$ factors as
$\Sigma^{ -n+1 }u'' \stackrel{ \delta }{ \rightarrow } \Sigma u
\rightarrow c$, but $\Hom_{ \Db }( \Sigma u,c ) = 0$ since
$u$ and $c$ are in $\mod( \Lambda )$.  We conclude $\sigma = 0$
as desired.
\end{proof}

\section{Torsion classes in $( n+2 )$-angulated categories associated to
  $n$-representation finite algebras}
\label{sec:n-angulated}

Recall that throughout, $n \geqslant 1$ is a fixed integer and we are
working under Setup \ref{set:blanket}.  The following is Theorem
\ref{thm:A} from the introduction.  The notion of being left closed
under $n$-extensions was defined after Theorem
\ref{thm:n-angulated_torsion2}.

\begin{Theorem}
\label{thm:n-angulated_torsion}
Let $\cX \subseteq \cC$ be a full subcategory closed under sums and
summands.  Then $\cX$ is a torsion class in $\cC$ $\Leftrightarrow$
$\cX$ is left closed under $n$-extensions.
\end{Theorem}

\begin{proof}
``$\Rightarrow$'': If $n=1$ then ``torsion class'' has the usual
meaning, see \cite[def.\ 2.2]{IY}, and ``closed under $1$-extensions''
means ``closed under extensions'', so the result holds by \cite[the
lines after def.\ 2.2]{IY}.

Assume $n \geqslant 2$.  Given a morphism $x'' \stackrel{ \delta }{
  \rightarrow } \Sigma^n x'$ with $x',x'' \in \cX$, we can complete to
an $( n+2 )$-angle in $\cC$,
\[
  x' \stackrel{ \xi' }{ \rightarrow }
  c^1 \stackrel{ \gamma }{ \rightarrow }
  c^2 \rightarrow \cdots \rightarrow c^n
  \rightarrow x'' \stackrel{ \delta }{ \rightarrow } \Sigma^n x',
\]
where $\gamma$ can be assumed to be in the radical of $\cC$ by
\cite[lem.\ 5.18 and its proof]{OT}.  We will show $c^1 \in \cX$.

Since $\cX$ is a torsion class, there is an $( n+2 )$-angle
\[
  x
  \stackrel{ \xi }{ \rightarrow } c^1
  \rightarrow d^2
  \rightarrow \cdots
  \rightarrow d^n 
  \stackrel{ \psi }{ \rightarrow } d^{ n+1 }
  \rightarrow \Sigma^n x
\]
in $\cC$ with $x \in \cX$ and $d^2 \rightarrow \cdots \rightarrow d^{
  n+1 }$ in $\Xexact$.  By Lemma \ref{lem:precovers}(ii) the morphism
$\xi'$ factors through $\xi$ so we can use axiom (F3) of
\cite[def.\ 2.1]{GKO} to get the following commutative diagram.
\[
  \xymatrix {
    x' \ar^{\xi'}[r] \ar[d] & c^1 \ar^{\gamma}[r] \ar@{=}[d]
      & c^2 \ar[r] \ar[d] & \cdots \ar[r] & c^n \ar[r] \ar[d]
      & x'' \ar^{\delta}[r] \ar^{\xi''}[d] 
      & \Sigma^n x' \ar[d] \\
    x \ar_{\xi}[r] & c^1 \ar[r] & d^2 \ar[r] & \cdots \ar[r]
      & d^n \ar_{ \psi }[r] & d^{n+1} \ar[r] & \Sigma^n x
            }
\]

Since $d^2 \rightarrow \cdots \rightarrow d^{ n+1 }$ is in
$\Xexact$, the morphism $\xi''$ factors through $\psi$.  Combining
this with Lemma \ref{lem:exact2}, we can construct a
chain homotopy as follows.
\[
  \xymatrix @C=3.5pc @R=1.75pc { 
    x' \ar^{\xi'}[r] \ar[d] & c^1 \ar^{\gamma}[r] \ar@{=}[d] \ar_{\sigma^1}[dl]
      & c^2 \ar[r] \ar[d] \ar_{\sigma^2}[dl] & \cdots \ar[r] & c^n \ar[r] \ar[d]
      & x'' \ar^{\xi''}[d] \ar_{\sigma^{n+1}}[dl] \\
    x \ar_{\xi}[r] & c^1 \ar[r] & d^2 \ar[r] & \cdots \ar[r]
      & d^n \ar_{ \psi }[r] & d^{n+1}
                     }
\]
In particular, $\id_{ c^1 } = \sigma^2\gamma + \xi\sigma^1$ whence
$\xi\sigma^1 = \id_{ c^1 } - \sigma^2\gamma$.  The right hand side is
invertible because $\gamma$ is in the radical of $\cC$, so $\sigma^1$
is a split injection and $c^1 \in \cX$ follows.

``$\Leftarrow$'':  In view of Lemma \ref{lem:functorially_finite}(ii),
this holds by Theorem \ref{thm:n-angulated_torsion2}.
\end{proof}

\begin{Example}
\label{exa:extension_closed}
We could also consider the condition that $\cX$ is {\em closed under
  $n$-extensions}, meaning that, for each morphism $x'' \stackrel{
  \delta }{ \rightarrow } \Sigma^n x'$ with $x', x'' \in \cX$, there
is an $( n+2 )$-angle 
$
  x'
  \rightarrow c^1
  \rightarrow c^2
  \rightarrow \cdots
  \rightarrow c^n
  \rightarrow x''
  \stackrel{ \delta }{ \rightarrow }
  \Sigma^n x'
$
with {\em each} $c^i$ in $\cX$.  However, this is strictly stronger
than $\cX$ being left closed under $n$-extensions.  For instance, let
$Q$ denote the quiver $1 \stackrel{ \alpha }{ \longrightarrow } 2
\stackrel{ \beta }{ \longrightarrow } 3$ and let $\Gamma = kQ / (
\beta\alpha )$.  The
Auslander--Reiten quiver of $\mod( \Gamma )$ is the following, where
indecomposable modules are denoted by the vertices where they are
supported. 
\[
  \xymatrix @-1pc {
    & {\mbox{$\begin{array}{c} 2 \\[-1mm] 3 \end{array}$}} \ar[dr] & & \mbox{$\begin{array}{c} 1 \\[-1mm] 2 \end{array}$} \ar[dr] & \\
    3 \ar[ur] & & 2 \ar[ur] & & 1
                  }
\]
Note that $3$ is a projective module, $1$ an injective module, while
$\begin{array}{c} 2 \\[-1mm] 3 \end{array}$ and $\begin{array}{c} 1
  \\[-1mm] 2 \end{array}$ are projective and injective.  The algebra
$\Gamma$ is weakly $2$-representation finite with $2$-cluster til\-ting
subcategory
$
  \cM( \Gamma ) = \add \{\,
  3,
  \begin{array}{c} 2 \\[-1mm] 3 \end{array},
  \begin{array}{c} 1 \\[-1mm] 2 \end{array},
  1
                        \,\}
$
inside $\mod( \Gamma )$.  We have $\gldim( \Gamma ) = 2$ and
$
  \cC( \Gamma ) = \add\,\{\, \Sigma^{2i}\cM( \Gamma ) \,|\, i \in \BZ \,\}
$
is a $2$-cluster tilting subcategory of $\Db( \mod\,\Gamma )$; see
Setup \ref{set:blanket}.  

It is not hard to check that
$\cX = \add \{\, 3, \begin{array}{c} 2 \\[-1mm] 3 \end{array}, 1 \,\}$,
viewed as a subcategory of $\cC( \Gamma )$, is left closed under
$2$-extensions.  However, it is not closed under $2$-extensions
because the $2$-extension represented by the $2$-exact sequence
$
  0
  \rightarrow 3
  \rightarrow \begin{array}{c} 2 \\[-1mm] 3 \end{array}
  \rightarrow \begin{array}{c} 1 \\[-1mm] 2 \end{array}
  \rightarrow 1
  \rightarrow 0
$
gives the $4$-angle
$
  3
  \rightarrow \begin{array}{c} 2 \\[-1mm] 3 \end{array}
  \rightarrow \begin{array}{c} 1 \\[-1mm] 2 \end{array}
  \rightarrow 1
  \rightarrow \Sigma^2( 3 )
$
in $\cC( \Gamma )$, and there is no way to get rid of 
$\begin{array}{c} 1 \\[-1mm] 2 \end{array}$ as a direct
summand of the third term.
\hfill $\Box$
\end{Example}

\section{Bijection between intermediate aisles and
  $n$-abelian torsion classes}
\label{sec:correspondence}

Recall that throughout, $n \geqslant 1$ is a fixed integer and we are
working under Setup \ref{set:blanket}.  This section proves Theorem
\ref{thm:C} from the introduction, see Theorem \ref{thm:intermediate}.

\begin{Definition}
\label{def:C_leq}
For $\ell \in \BZ$ set
$
  \cC^{ \leqslant \ell n }
  = \add \{\, \Sigma^{ in }\cM \,|\, i \geqslant -\ell \,\}.
$
In other words, $\cC^{ \leqslant \ell n }$ is the part of $\cC$ in
cohomological degrees $\leqslant \ell n$.
\hfill $\Box$
\end{Definition}

\begin{Remark}
It is not hard to check directly that each $\cC^{\leqslant \ell n}$
is an aisle in $\cC$.  This is also a consequence of Theorem
\ref{thm:intermediate} below.  \hfill $\Box$
\end{Remark}

\begin{Definition}
\label{def:XU}
If $\cU \subseteq \cM$ is a full subcategory then set
$
  \cX( \cU ) = \add \bigl( \cU \cup \cC^{ \leqslant -n } \bigr).
$
\hfill $\Box$
\end{Definition}

\begin{Remark}
By construction, $\cX( \cU )$ is ``intermediate'' in the sense $\cC^{
  \leqslant -n } \subseteq \cX( \cU ) \subseteq \cC^{ \leqslant 0}$.  Moreover,
$\Sigma^n \cX( \cU ) \subseteq \cX( \cU )$, so $\cX( \cU )$ has a shot
at being what we can reasonably call an intermediate aisle.
\hfill $\Box$
\end{Remark}

Indeed, the following is the main result of this section.

\begin{Theorem}
\label{thm:intermediate}
The assignment $\cU \mapsto \cX( \cU )$ defines a bijection
\[
  \biggl\{
    \cU \subseteq \cM
  \,\bigg|\,
      \cU \mbox{ is a torsion class}
  \biggr\}
  \rightarrow
  \biggl\{
    \cX \subseteq \cC
  \,\bigg|
    \begin{array}{ll}
      \mbox{$\cX$ is an aisle with} \\
      \cC^{\leqslant -n} \subseteq \cX \subseteq \cC^{\leqslant 0} \\
    \end{array}
  \biggr\}.
\]
\end{Theorem}

\begin{proof}
It is clear that $\cU \mapsto \cX( \cU )$ defines a bijection 
\[
  \Biggl\{
    \cU \subseteq \cM
  \,\Bigg|
    \begin{array}{ll}
      \cU \mbox{ is a full subcategory} \\
      \mbox{closed under sums and} \\
      \mbox{summands}
    \end{array}
  \Biggr\}
  \rightarrow
  \Biggl\{
    \cX \subseteq \cC
  \,\Bigg|
    \begin{array}{ll}
      \mbox{$\cX$ is a full subcategory closed} \\
      \mbox{under sums and summands} \\
      \mbox{with }
      \cC^{\leqslant -n} \subseteq \cX \subseteq \cC^{\leqslant 0}
    \end{array}
  \Biggr\}.
\]
To prove the theorem, we must hence let $\cU \subseteq \cM$ be a full
subcategory closed under sums and summands and show that $\cU$ is a
torsion class if and only if $\cX( \cU )$ is an aisle.
Since $\Sigma^n \cX( \cU ) \subseteq \cX( \cU )$ holds by
construction, it is enough to show that $\cU$ is a torsion class if
and only if $\cX( \cU )$ is a torsion class.  By Theorems
\ref{thm:n-abelian_torsion} and \ref{thm:n-angulated_torsion} this
amounts to 
\begin{equation}
\label{equ:tors_eq}
  \mbox{
    $\cU$ has Property $(E)$ $\Leftrightarrow$ 
    $\cX( \cU )$ is left closed under $n$-extensions,
       }
\end{equation}
and we will show the two implications.

``$\Rightarrow$'' in Equation \eqref{equ:tors_eq}: Let $x'' \stackrel{
  \delta }{ \rightarrow } \Sigma^n x'$ be a morphism with $x',x'' \in
\cX( \cU )$.  We must show that there is an $( n+2 )$-angle in $\cC$,
\[
  x'
  \rightarrow c^1
  \rightarrow c^2
  \rightarrow \cdots
  \rightarrow c^n
  \rightarrow x''
  \stackrel{ \delta }{ \rightarrow }
  \Sigma^n x',
\]
with $c^1 \in \cX( \cU )$.  According to the construction in
\cite[proof of thm.\ 1]{GKO}, we can get $c^1$ by constructing a
triangle
\begin{equation}
\label{equ:triangle}
  \Sigma^{ -n }x''
  \stackrel{ \Sigma^{ -n }\delta }{ \longrightarrow } x'
  \longrightarrow d
  \longrightarrow \Sigma^{ -n+1 }x''
\end{equation}
in $\Db( \mod\,\Lambda )$ and letting $d \rightarrow c^1$ be a
$\cC$-preenvelope.  Note that because
\[
  \cC = \add \{\, \Sigma^{ in }\cM \,|\, i \in \BZ \,\},
\]
we can let
$
  c^1 = \bigoplus_{i \in \BZ} \Sigma^{ in } m( i )
$
where $m( i ) \in \cM$ and $d \rightarrow \Sigma^{ in } m( i )$ is a
$(\Sigma^{ in }\cM)$-preenvelope for each $i$.  To show that we can
get $c^1 \in \cX( \cU )$, it is then enough to show that we can let
\begin{enumerate}

  \item  $m( i ) = 0$ for $i \leqslant -1$,

\smallskip

  \item  $m( 0 ) \in \cU$.

\end{enumerate}

Here (i) can be achieved because if $i \leqslant -1$ then $\Hom_{ \Db
}( d,\Sigma^{ in }\cM ) = 0$ for degree reasons.  Specifically, the
long exact cohomology sequence of the triangle \eqref{equ:triangle}
shows that the cohomology of $d$ is concentrated in degrees $\leqslant
n-1$, while the cohomology of an object of $\Sigma^{ in }\cM$ is
concentrated in degree $-in \geqslant n$.

To achieve (ii) we must show that there is an $\cM$-preenvelope $d
\rightarrow m(0)$ with $m(0) \in \cU$.  By definition of $\cX( \cU )$
we have $x' = u' \oplus c'$ and $x'' = u'' \oplus c''$ with $u',u''
\in \cU$ and $c',c'' \in \cC^{ \leqslant -n }$.  The octahedral axiom gives
a commutative diagram where the rows and columns are triangles
minus the fourth object, the second vertical triangle is split, and
the second horizontal triangle is \eqref{equ:triangle}.
\[
  \xymatrix { 
    0 \ar[r] \ar[d] & c' \ar@{=}[r] \ar[d] & c' \ar[d] \\
    \Sigma^{ -n }( u'' \oplus c'' ) \ar^-{ \Sigma^{ -n }\delta }[r] \ar@{=}[d] & u' \oplus c' \ar[r] \ar[d] & d \ar[d] \\
    \Sigma^{ -n }( u'' \oplus c'' ) \ar_-{ \Sigma^{ -n }\delta' }[r] & u' \ar[r] & d'
            }
\]
Here $\Hom_{ \Db }( c',\cM ) = 0$ for degree reasons, so
any morphism from $d$ to an object of $\cM$ factors through $d'$, so
it is enough to see that there is an $\cM$-preenvelope $d' \rightarrow
m(0)$ with $m(0) \in \cU$.

Since $c'' \in \cC^{ \leqslant -n}$ we have $\Sigma^{ -n }c'' \in \cC^{
  \leqslant 0}$ so we can write $\Sigma^{ -n }c'' = m'' \oplus
\widetilde{c}$ with $m'' \in \cM$ and $\widetilde{c} \in \cC^{ \leqslant -n
}$, whence $\Sigma^{ -n }( u'' \oplus c'' ) = \Sigma^{ -n }u'' \oplus
m'' \oplus \widetilde{c}$.  We have $\Hom_{ \Db }(
\widetilde{c},u' ) = 0$ for degree reasons, so the octahedral axiom
gives a commutative diagram where the rows and columns are
triangles minus the fourth object, the first vertical triangle is
split, and the second horizontal triangle is the third horizontal
triangle from the previous diagram.
\[
  \xymatrix { 
    \widetilde{c} \ar[r] \ar[d] & 0 \ar[r] \ar[d] & \Sigma \widetilde{c} \ar[d] \\
    \Sigma^{ -n }u'' \oplus m'' \oplus \widetilde{c} \ar^-{\Sigma^{ -n }\delta' }[r] \ar[d] & u' \ar[r] \ar@{=}[d] & d' \ar[d] \\
    \Sigma^{ -n }u'' \oplus m'' \ar[r] & u' \ar[r] & d''
            }
\]
Here $\Hom_{ \Db }( \Sigma \widetilde{c},\cM ) = 0$ for degree
reasons, so as above it is enough to see that there is an
$\cM$-preenvelope $d'' \rightarrow m(0)$ with $m(0) \in \cU$.  But
this is true by the third horizontal triangle because $\cU$ has
Property $(E)$.

``$\Leftarrow$'' in Equation \eqref{equ:tors_eq}: Consider the
triangle of Definition \ref{def:E},
\[
  \Sigma^{ -n }u'' \oplus m''
  \stackrel{ \psi'' }{ \rightarrow } u'
  \stackrel{ \psi' } { \rightarrow } e
  \stackrel{ \psi } { \rightarrow } \Sigma^{ -n+1 }u'' \oplus \Sigma m''
\]
with $u',u'' \in \cU$ and $m'' \in \cM$.  We must show that there is
an $\cM$-preenvelope $e \rightarrow u$ with $u \in \cU$.

The subcategory $\cX( \cU )$ of $\cC$ is left closed under
$n$-extensions, and since $u'$ and $u'' \oplus \Sigma^n m''$ are in
$\cX( \cU )$, there is an $( n+2 )$-angle in $\cC$,
\[
  u'
  \stackrel{ \theta' }{ \longrightarrow } x^1
  \longrightarrow y^2
  \longrightarrow \cdots
  \longrightarrow y^n
  \longrightarrow u'' \oplus \Sigma^n m''
  \stackrel{ \Sigma^n \psi'' }{ \longrightarrow } \Sigma^n u',
\]
with $x^1 \in \cX( \cU )$.  By definition of $\cX( \cU )$ we have $x^1
= u \oplus c$ with $u \in \cU$ and $c \in
\cC^{ \leqslant -n}$.  Let $x^1 \stackrel{ \xi }{ \rightarrow } u$ be the projection onto the summand.  It follows from
Lemma \ref{lem:exact2} that $\theta' \circ \psi'' = 0$, so we
get the following commutative diagram.
\[
  \xymatrix { 
    \Sigma^{ -n }u'' \oplus m'' \ar^-{\psi''}[r] \ar@{=}[d] & u' \ar^{\psi'}[r] \ar@{=}[d] & e \ar^{\varphi}[d] \ar^-{\psi}[r] & \Sigma^{ -n+1 }u'' \oplus \Sigma m'' \\
    \Sigma^{ -n }u'' \oplus m'' \ar_-{\psi''}[r] & u' \ar_{\theta'}[r] & x^1 \ar^{\xi}[d] \\
    & & u
            }
\]
We will show that $e \stackrel{ \xi\varphi }{ \longrightarrow }
u$ is an $\cM$-preenvelope.

Let $e \stackrel{ \chi }{ \rightarrow } m$ be a morphism
with $m \in \cM$.  We must show that it factors through $e
\stackrel{ \xi\varphi }{ \longrightarrow } u$.  Since
$\chi\psi' \circ \psi'' = \chi \circ \psi'\psi'' = \chi \circ 0 = 0$,
the dual of Lemma \ref{lem:exact2} implies that there is $x^1
\stackrel{ \alpha }{ \rightarrow } m$ with $\alpha\theta' =
\chi\psi'$.  Hence $( \chi - \alpha\varphi )\psi' = \chi\psi' -
\alpha\theta' = 0$ so there is $\Sigma^{ -n+1 }u'' \oplus \Sigma m''
\stackrel{ \beta }{ \rightarrow } m$ with $\beta\psi =
\chi - \alpha\varphi$.  However, this implies
\begin{equation}
\label{equ:fact}
  \chi = \alpha\varphi
\end{equation}
because $\beta = 0$ since $\Hom_{ \Db }( \Sigma^{ -n+1
}u'' \oplus \Sigma m'' , m ) = 0$.  Namely,
$
  \Hom_{ \Db }( \Sigma^{ -n+1 }u'',m )
  \cong \Ext_{ \Lambda }^{ n-1 }( u'',m ) = 0
$
because $u''$ and $m$ are in the $n$-cluster tilting subcategory
$\cM$, and
$
  \Hom_{ \Db }( \Sigma m'',m ) = 0
$
because $m''$ and $m$ are in $\cM$, hence in $\mod( \Lambda )$.

Now consider $x^1 \stackrel{ \alpha }{ \rightarrow } m$.  Here $x^1 =
u \oplus c$ with $c \in \cC^{ \leqslant -n }$, and $\Hom_{ \Db }( c,m
) = 0$ for degree reasons, so $\alpha$ factors through the projection
$x^1 \stackrel{ \xi }{ \rightarrow } u$.  That is, $\alpha =
\gamma\xi$, and combining with Equation \eqref{equ:fact} shows $\chi =
\gamma \circ \xi\varphi$.  So $\chi$ factors through $\xi\varphi$ as
desired.
\end{proof}

\section{Example: Splitting torsion classes arising from $n$-APR
  tilting modules} 
\label{sec:example1}

Recall that throughout, $n \geqslant 1$ is a fixed integer and we are
working under Setup \ref{set:blanket}.

In classic tilting theory, if $t$ is a tilting module then $\Fac(t)$,
the set of all quotients of modules of the form $t^{ \oplus d }$, is a
torsion class.  If $t$ is a so-called APR tilting module, see
\cite[thm.\ 1.11]{APR} or \cite[exa.\ VI.2.8(c)]{ASS}, then $\Fac(t)$
is splitting in the sense that for each $m \in \mod( \Lambda )$, the
short exact sequence $0 \rightarrow u \rightarrow m \rightarrow v
\rightarrow 0$ with $u \in \Fac(t)$, $v \in \Fac(t)^{ \perp }$ is
split.  We will show an $n$-homological version of this.

\begin{Setup}
\label{set:nAPR}
Let $p$ be a simple projective, non-injective left $\Lambda$-module
and write $\Lambda = p \oplus q$ as left $\Lambda$-modules.  The
corresponding $n$-APR tilting module, introduced in \cite[def.\ 3.1
and obs.\ 4.1(1)]{IO}, is $t = ( \tau^-_n p ) \oplus q$ where
$\tau^-_n = \Tr \Omega^{n-1} \dual$ is the composition of the
transpose, $\Tr$, the $(n-1)$st syzygy in a minimal projective
resolution, $\Omega^{n-1}$, and $k$-linear duality, $\dual$.  See
\cite[sec.\ IV.2]{ASS} for the definition of the transpose.  \hfill
$\Box$
\end{Setup}

For the following theorem, recall that splitting torsion classes were
defined after Theorem \ref{thm:06}.

\begin{Theorem}
\label{thm:Fac}
The full subcategory $\cM \cap \Fac(t)$ is a splitting torsion class
in $\cM$.
\end{Theorem}

\begin{proof}
We have $\cM = \add(M)$ for a basic $n$-cluster tilting module $M
\in \mod( \Lambda )$.  The indecomposable projective module $p$ is a
direct summand of $M$, see \cite[def.\ 1.1]{I1} or
Definition \ref{def:cluster_tilting}, so we can write
$M = M' \oplus p$ where $p$ is not a direct summand of $M'$.
We first remark that
\[
  \cM \cap \Fac(t) = \add( M' ).
\]
The inclusion $\supseteq$ follows from \cite[lem.\ 4.4]{IO} and the
second bullet in \cite[lem.\ 3.5]{IO}, and the inclusion $\subseteq$
follows from \cite[prop.\ 3.3(2)]{IO} which implies
\[
  \Hom_{ \Lambda }( t,p ) = 0,
\]
whence the indecomposable $p$ is not in $\cM \cap \Fac(t)$.

To complete the proof, we show that $\add( M' )$ is a splitting
torsion class.  For $m \in \cM$ there is a split short exact sequence
$0 \rightarrow u \rightarrow m \rightarrow v \rightarrow 0$ with $u
\in \add( M' )$ and $v \in \add( p )$.  It gives an $n$-exact sequence
$
  0 
  \rightarrow u
  \stackrel{\theta}{\rightarrow} m
  \rightarrow v
  \rightarrow 0
  \rightarrow \cdots
  \rightarrow 0
  \rightarrow 0
$
in $\cM$ with $\theta$ a split monomorphism.  Here $v \rightarrow 0
\rightarrow \cdots \rightarrow 0$ is in $\add( M' )\mbox{\rm -exact}$
because $\Hom_{ \Lambda }( M',p ) = 0$; this follows from the first
two displayed equations of the proof.
\end{proof}

\begin{Remark}
The $n$-APR tilting module $t$ is in $\cM$ by \cite[thm.\ 4.2(1)]{IO},
it is a tilting module in the sense of Miyashita \cite[p.\ 113]{M} by
\cite[thm.\ 3.2]{IO}, and $\gldim \End( t ) = n$ by \cite[prop.\
3.6]{IO}.  A more careful analysis shows that if $t$ satisfies these
conditions, then Theorem \ref{thm:Fac} remains true.  However, it is
presently unclear what can be said if $t \in \cM$ is a tilting module
with $\gldim \End( t ) \geqslant n+1$.
\hfill $\Box$
\end{Remark}

\section{Example: A non-splitting torsion class} 
\label{sec:example2}

Recall that throughout, $n \geqslant 1$ is a fixed integer and we are
working under Setup \ref{set:blanket}.  We keep Setup \ref{set:nAPR}
in place.

\begin{Theorem}
The full subcategory $\cU = \add( p )$ is a non-splitting torsion
class in $\cM$.
\end{Theorem}

\begin{proof}
By Theorem \ref{thm:n-abelian_torsion} and Lemma \ref{lem:EeqF}, it is
enough to show that $\cU$ has Property $(F)$.

Since $p$ is indecomposable, each module $u \in \cU$ has the form $u
= p^{ \oplus i }$.  Since $p$ is simple, each quotient of $u$ in
$\mod( \Lambda )$ has the form $p^{ \oplus j }$.  Hence, in the
triangle
$
  \Sigma^{ -n }u''
  \stackrel{ \varphi }{ \rightarrow } b'
  \rightarrow f
  \rightarrow \Sigma^{ -n+1 }u''
$
of Definition \ref{def:F}, 
we have $u'' = p^{ \oplus i }$ and $b' = p^{ \oplus j }$.  Since $p$
is projective and $n \geqslant 1$ we have $\varphi = 0$ so
$
  f = b' \oplus \Sigma^{ -n+1 }u''
    = p^{ \oplus j } \oplus \Sigma^{ -n+1 }p^{ \oplus i }.
$

If $n = 1$ then $f = p^{ \oplus ( i+j ) }$.  This is an object of $\cU$,
so $f$ has an $\cM$-preenvelope in $\cU$, namely $f$ itself.  If $n
\geqslant 2$ then $\Hom_{ \Db }( \Sigma^{ -n+1 }p^{ \oplus i },\cM ) =
0$ for degree reasons, so to get an $\cM$-preenvelope of $f$ we can
take one of $p^{ \oplus j }$.  But $p^{ \oplus j }$ is an object of
$\cU$, so $f$ has an $\cM$-preenvelope in $\cU$, namely $p^{ \oplus j
}$.  This verifies Property $(F)$.

Finally, to show that $\cU$ is non-splitting, let $E( p )$ be the
injective envelope of $p$ in $\mod( \Lambda )$.  Then $E( p ) \in
\cM$, see \cite[def.\ 1.1]{I1} or Definition
\ref{def:cluster_tilting}, so there is an $n$-exact sequence
\[
  0 \rightarrow u \stackrel{ \theta }{ \rightarrow } E( p )
    \rightarrow v^1 \rightarrow \cdots \rightarrow v^n \rightarrow 0
\]
as in Definition \ref{def:torsion}.  By Lemma \ref{lem:precovers}(i)
the morphism $\theta$ is a $\cU$-cover so must be non-zero since the
object $p \in \cU$ has a non-zero morphism $p \rightarrow E( p )$.  If
$\cU$ were splitting then we could assume that $\theta$ was split
whence $u$ would be injective.  Since $u = p^{ \oplus i }$, this would
show $p$ injective contradicting Setup \ref{set:nAPR}.
\end{proof}

\medskip
\noindent
{\bf Acknowledgement.}  We thank the referees for their careful
reading of the paper and their detailed comments and suggestions which
have led to significant improvements.

\end{document}